\documentclass{elsarticle}

\usepackage{amscd}
\usepackage{amsmath}
\usepackage{amssymb}
\usepackage{amsthm}
\usepackage{graphicx}
\usepackage[all]{xy}
\usepackage{kotex}
\usepackage{color}
\usepackage{enumerate}
\theoremstyle{plain}

\usepackage{nameref}

\newtheorem{Theorem}{Theorem}[section]

\newtheorem{Cor}[Theorem]{Corollary}
\newtheorem{definition}[Theorem]{Definition}
\newtheorem*{theorem*}{Theorem}
\theoremstyle{remark}
\newtheorem{remark}[Theorem]{\bf{Remark}}
\newtheorem{example}[Theorem]{\bf{Example}}

\usepackage{color}							

\journal{Arxiv}
\begin{document}
	\begin{frontmatter}
		\title{A Necessary and Sufficient Condition for the Existence of Global Solutions to Semilinear Parabolic Equations on bounded domains}
		
		\author[a]{Soon-Yeong Chung\corref{cjhp}}
		\ead{sychung@sogang.ac.kr}
		\author[b]{Jaeho Hwang}
		\ead{hjaeho@sogang.ac.kr}

		\cortext[cjhp]{Corresponding author}
		\address[a]{Department of Mathematics and Program of Integrated Biotechnology, Sogang University, Seoul 04107, Republic of Korea}
		\address[b]{Research Institute for Basic Science, Sogang University, Seoul 04107, Republic of Korea}

		\begin{abstract}
			The purpose of this paper is to give a necessary and sufficient condition for the existence and non-existence of global solutions of the following semilinear parabolic equations
			\[
			u_{t}=\Delta u+\psi(t)f(u),\,\,\mbox{ in }\Omega\times (0,t^{*}),
			\]
			under the Dirichlet boundary condition on a bounded domain. In fact, this has remained as an open problem for a few decades, even for the case $f(u)=u^{p}$. As a matter of fact, we prove:\\
			\[
			\begin{aligned}
			&\mbox{there is no global solution for any initial data if and only if }\\
			&\mbox{the function } f \mbox{ satisfies}\\
			&\hspace{20mm}\int_{0}^{\infty}\psi(t)\frac{f\left(\lVert S(t)u_{0}\rVert_{\infty}\right)}{\lVert S(t)u_{0}\rVert_{\infty}}dt=\infty\\
			&\mbox{for every }\,\epsilon>0\,\mbox{ and nonnegative nontrivial initial data }\,u_{0}\in C_{0}(\Omega).
			\end{aligned}
			\]
			Here, $(S(t))_{t\geq 0}$ is the heat semigroup with the Dirichlet boundary condition.
		\end{abstract}
		
	\end{frontmatter}

\section{Introduction}
In his seminal paper \cite{F}, Fujita firstly studied the reaction-diffusion equation
\[
u_{t}=\Delta u+u^{p},\,\,\mbox{ in }\mathbb{R}^{N}\times (0,t^{*}),
\]
where $p>1$, and obtained that
\begin{itemize}
	\item[(i)] if $1<p<p^{*}$, then there is no global solution for any initial data,
	\item[(ii)] if $p>p^{*}$, then there exists a global solution whenever the initial data is sufficiently small,
\end{itemize}
where $p^{*}=1+\frac{2}{N}$ is called the critical exponent. After his results, researchers obtained that there is no global solution for $p=p^{*}$ (see \cite{Ha} for the case $N=1$ or $2$ and \cite{AW} for the case $N\geq 3$).

It is easy to see that the critical exponent $p^{*}$ leads a necessary and sufficient condition for the existence of the global solutions as above. Therefore, lots of researchers have been studied the critical exponent for various reaction-diffusion equations to find necessary and sufficient conditions for the existence of the global solutions (see the survey articles \cite{DL,L}).\\

In this paper, we discuss the existence and nonexistence of the global solutions to the reaction-diffusion equation for a general source term $\psi(t)f(u)$:
\begin{equation}\label{mainequation}
	\begin{cases}
		u_{t}(x,t)=\Delta u(x,t)+\psi(t)f(u(x,t)),\,\,&(x,t)\in \Omega\times(0,t^{*}),\\
		u(x,t)=0,&(x,t)\in \partial\Omega\times(0,t^{*}),\\
		u(x,0)=u_{0}(x)\geq 0,&x\in \Omega,
	\end{cases}
\end{equation}
where $\Omega$ is a bounded domain in $\mathbb{R}^{N}$ with a smooth boundary $\partial\Omega$, $u_{0}$ is a nonnegative and nontrivial $C_{0}(\Omega)$-function, and $\psi$ is a nonnegative continuous function on $[0,\infty)$. Throughout this paper, we always assume that $f$ satisfies the following conditions:
\begin{itemize}
	\item[$\cdot$] $f:[0,\infty)\rightarrow [0,\infty)$ is a locally Lipschitz continuous function,
	\item[$\cdot$] $f(0)=0$ and $f(s)>0$ for $s>0$,
\end{itemize}
Then it is well-known that the local existence of the solutions of the equation \eqref{mainequation} and the comparison principle are guaranteed by the first condition.\\

In his pioneering paper \cite{M}, Meier studied the global existence and nonexistence of the solutions to the reaction-diffusion equations \eqref{mainequation}, where $\Omega$ is a general (bounded or unbounded) domain in $\mathbb{R}^{N}$, under the Dirichlet boundary condition and obtained the following result:
\begin{Theorem}[\cite{M}]\label{Meier}
	Assume that $\psi\in C[0,\infty)$ and $f(u)=u^{p}$ for $p>1$.
	\begin{itemize}
		\item[(i)] If $\limsup_{t\rightarrow \infty}\lVert S(t)u_{0}\rVert_{\infty}^{p-1}\int_{0}^{t}\psi(\tau)d\tau = \infty$ for every $u_{0}\in C_{0}(\Omega)$, then there is no global solution for any nonnegative and nontrivial initial data.
		\item[(ii)] If $\int_{0}^{\infty}\psi(\tau)\lVert S(\tau)u_{0}\rVert_{\infty}^{p-1}d\tau <\infty$ for some $u_{0}\in C_{0}(\Omega)$, then there exists global solution for sufficiently small initial data.
	\end{itemize}
	Here, $(S(t))_{t\geq 0}$ is the heat semigroup with the Dirichlet boundary condition.
\end{Theorem}
Meier give two sufficient conditions for the existence and nonexistence of the global solutions. However, a necessary and sufficient condition for the existence of the global solutions to the equation \eqref{mainequation} is unknown, even for the case $f(u)=u^{p}$ and has remained as an open problem for a few decades. To our best knowledge, researches of necessary and sufficient conditions for the global existence of solutions for the reaction-diffusion equations in the current literature consider several specific source terms such as $t^{\sigma}u^{p}$, $e^{\beta t}u^{p}$, and so on (see \cite{M,BZW,Q,QW}).

Recent researches of the equation \eqref{mainequation} have adopted Meier's criterion and give several sufficient conditions for the blow-up solutions and global solutions (for example, see \cite{LP,CL1,CL2}). In conclusion, the open problem has faced methodological limitations and there has been no progress in research on necessary and sufficient conditions for the general source term $\psi(t)f(u)$.\\

From the above point of view, the purpose of this paper is twofold as follows:
\begin{itemize}
	\item[(i)] to obtain the necessary and sufficient condition for the existence of the global solutions for more general source term $\psi(t)f(u)$.
	\item[(ii)] to introduce a method, so called a minorant method, to deal with $f(u)$ in the source term.
\end{itemize}

In order to solve the open problem mentioned above for more general source term $f(u)$ instead of $u^{p}$, we need to introduce the minorant function $f_{m}$ and the majorant function $f_{M}$ to overcome a multiplicative property as follows:
\[
f_{m}(u):=\inf_{0<\alpha<1}\frac{f(\alpha u)}{f(\alpha)}\,\mbox{ and }\,f_{M}(u):=\sup_{0<\alpha<1}\frac{f(\alpha u)}{f(\alpha)}.
\]
The minorant function $f_{m}$ and the majorant function $f_{M}$ were firstly discussed in \cite{CH}. By the definition, the function $f$ satisfies the following inequality:
\[
f_{m}(u)f(\alpha)\leq f(u \alpha)\leq f_{M}(u)f(\alpha),\,\,\,u>0,\,\,\,0<\alpha<1.
\]
Then the following condition is the necessary condition for blow-up solutions in view of the minorant method (see \cite{Fr,H}):
\begin{equation}\label{assum}
	\int_{1}^{\infty}\frac{ds}{f_{m}(s)}<\infty.
\end{equation}

Finally, we obtained the following results to see `completely' whether or not we have global solutions:
\begin{Theorem}\label{mainthm1}
	Let $f$ be a convex function satisfying the assumption \eqref{assum} and $\psi$ be a nonnegative continuous function. Then the following statements are equivalent.
	\begin{itemize}
		\item[(i)] there is no global solution $u$ to the equation \eqref{mainequation} for any nonnegative and nontrivial initial data $u_{0}$.
		\item[(ii)] \[
		\int_{0}^{\infty}\psi(t)e^{\lambda_{0}t}f\left(\epsilon e^{-\lambda_{0}t}\right)dt=\infty
		\]
		for every $\epsilon>0$.
		\item[(iii)] \[\int_{0}^{\infty}\psi(t)\frac{f\left(\lVert S(t)u_{0}\rVert_{\infty}\right)}{\lVert S(t)u_{0}\rVert_{\infty}}dt=\infty\]
		for every nonnegative nontrivial initial data $u_{0}\in C_{0}(\Omega)$.
	\end{itemize}
	Here, $(S(t))_{t\geq 0}$ is the heat semigroup with the Dirichlet boundary condition and $\lambda_{0}$ is the first Dirichlet eigenvalue of the Laplace operator $\Delta$.
\end{Theorem}
Theorem \ref{mainthm1} is the form of a necessary and sufficient condition for global solutions of the equation \eqref{mainequation}. Therefore, the open problem mentioned above is solved with more general source term $\psi(t)f(u)$.

In general, the case $p=p^{*}$ and $p\neq p^{*}$ are dealt in a different way and cannot be solved at the same time (see \cite{AW,F,Ha,H}). However, we prove the cases all at once.

As far as authors know, there is no paper which discuss the necessary and sufficient condition (or Fujita's blow-up solutions) on the source term $f(u)$ instead of $u^{p}$. From this point of view, this paper and the minorant method will be clue to study long-time behaviors (especially, necessary and sufficient conditions) for solutions to PDEs with general type functions.

We organized this paper as follows: In Section \ref{discussion}, we discuss Meier's criterion. We introduce the minorant method and discuss main results in Section \ref{main}.

\section{Discussion on Meier's conditions}\label{discussion}
The purpose of this section is to discuss the necessary and sufficient condition for the existence of the global solutions, which has not been known and remained as an open problem. Let us deal with sufficient conditions for the blow-up solutions and global solutions to check that these conditions can be a necessary and sufficient condition.

From this point of view, let us discuss Meier's conditions. If the domain $\Omega$ is bounded, then it is well-known that $\lVert S(t)u_{0}\rVert_{\infty} \sim e^{-\lambda_{0} t}$ for $t\geq 0$, for every nonnegative and nontrivial initial data $u_{0}\in C_{0}(\Omega)$. Therefore, Theorem \ref{Meier} can be understood as follows:
\begin{itemize}
	\item[$\cdot$] If
	\[
	(C1)\,:\,\limsup_{t\rightarrow \infty} e^{-(p-1)\lambda_{0}t}\int_{0}^{t}\psi(\tau) d\tau=\infty,
	\]
	then there is no global solution to the equation \eqref{mainequation} for any nonnegative and nontrivial initial data.
	\item[$\cdot$] If
	\[
	(C2)\,:\,\int_{0}^{\infty}\psi(t) e^{-(p-1)\lambda_{0}t}dt<\infty
	\]
	then there exists a global solution to the equation \eqref{mainequation} for sufficiently small initial data.
\end{itemize}
Let us consider the function $\psi$ defined by $\psi(t):=(t+1)^{-\delta}e^{(p-1)\lambda_{0} t}$ for $0\leq \delta\leq 1$. Then it follows that
\[
\begin{aligned}
	\limsup_{t\rightarrow \infty} e^{-(p-1)\lambda_{0}t}\int_{0}^{t}\psi(\tau) d\tau=&\limsup_{t\rightarrow \infty} e^{-(p-1)\lambda_{0}t}\int_{0}^{t}(\tau+1)^{-\delta}e^{(p-1)\lambda_{0} \tau} d\tau\\
	\leq&\limsup_{t\rightarrow \infty} e^{-(p-1)\lambda_{0}t}\int_{0}^{t}e^{(p-1)\lambda_{0} \tau} d\tau\\
	=&\frac{1}{(p-1)\lambda_{0}}<\infty
\end{aligned}
\]
and
\[
\begin{aligned}
	\int_{0}^{\infty}\psi(t) e^{-(p-1)\lambda_{0}t}dt=&\int_{0}^{\infty}(t+1)^{-\delta}dt=\infty.
\end{aligned}
\]
This implies that if the function $\psi(t):=(t+1)^{-\delta}e^{(p-1)\lambda_{0} t}$ for $0\leq \delta\leq 1$, then we don't know whether or not the solution exists globally.\\

Now, we are going to check whether or not the solution exists globally, by considering the simple example. Let's consider the functions $\psi$ and $f$ defined by $\psi(t):=(t+1)^{-\frac{1}{2}}e^{\lambda_{0}t}$ and $f(u):=u^{2}$ in the equation \eqref{mainequation}. Then the equation \eqref{mainequation} follows that
\begin{equation}\label{exeq}
	\begin{cases}
		u_{t}(x,t)=\Delta u(x,t)+(t+1)^{-\frac{1}{2}}e^{\lambda_{0}t}u^{2},\,\,&(x,t)\in \Omega\times(0,t^{*}),\\
		u(x,t)=0,&(x,t)\in \partial\Omega\times(0,t^{*}),\\
		u(x,0)=u_{0}(x)\geq 0,&x\in \Omega.
	\end{cases}
\end{equation}
Now, we consider the eigenfunction $\phi_{0}$ to be $\sup_{x\in\Omega}\phi_{0}dx=1$, corresponding to the first Dirichlet eigenvalue $\lambda_{0}$. Suppose that the solution $u$ to the equation \eqref{exeq} exists globally. Multiplying the equation \eqref{exeq} by $\phi_{0}$ and integrating over $\Omega$, we use Green's theorem and Jensen's inequality to obtain
\[
\begin{aligned}
	&\int_{\Omega}u_{t}(x,t)\phi_{0}(x)dx\\=&\int_{\Omega}\phi_{0}(x)\Delta u(x,t)dx+(t+1)^{-\frac{1}{2}}e^{\lambda_{0}t}\int_{\Omega}u^{2}\phi_{0}(x)dx\\
	=&-\lambda_{0}\int_{\Omega}u(x,t)\phi_{0}(x)dx+(t+1)^{-\frac{1}{2}}e^{\lambda_{0}t}\left(\int_{\Omega}u(x,t)\phi_{0}(x)dx\right)^{2},
\end{aligned}
\]
for all $t> 0$. Putting $y(t):=\int_{\Omega}u(x,t)\phi_{0}(x)dx$, for $t\geq 0$, then $y(t)$ exists for all time $t$ and satisfies the following inequality
\begin{equation}\label{3111}
	\begin{cases}
		y'(t)\geq -\lambda_{0} y(t)+(t+1)^{-\frac{1}{2}}e^{\lambda_{0}t}y^{2}(t),\,\,t>0,\\
		y(0)=y_{0}:=\int_{\Omega}u_{0}(x)\phi_{0}(x)dx>0.
	\end{cases}
\end{equation}
Multiplying $e^{\lambda_{0}t}$ by the inequality \eqref{3111}, then we have
\[
\left[e^{\lambda_{0}t}y(t)\right]'\geq (t+1)^{-\frac{1}{2}}\left[e^{\lambda_{0}t}y(t)\right]^{2}\geq 0,
\]
for all $t>0$, which implies that
\[
\frac{d}{dt}\left[e^{\lambda_{0}t}y(t)\right]^{-1}\leq -(t+1)^{-\frac{1}{2}}
\]
for all $t>0$. Solving the differential inequality, then we obtain that
\[
y(t)\geq \frac{e^{-\lambda_{0} t}}{y_{0}^{-1}-\int_{0}^{t}(\tau+1)^{-\frac{1}{2}}d\tau},
\]
for all $t>0$, which leads a contradiction. Hence, the solution $u$ to the equation \eqref{exeq} blows up at finite time.

The above example implies that the condition $(C1)$ is no longer necessary condition for the nonexistence of global solution. In fact, the main part of this paper is focused on the condition $(C2)$ to see whether $(C2)$ is necessary and sufficient condition of the existence of the global solution.\\

On the other hand, if the function $f$ in the equation \eqref{mainequation} doesn't have a multiplicative property, then we cannot apply Meier's results. For example, let us consider the function $f(u)=\frac{u^{2}+u}{2}$. Then it is easy to see that $u\leq f(u)\leq u^{2}$ for $u\geq 1$ and $u^{2}\leq f(u)\leq u$ for $0\leq u\leq 1$. Therefore, we cannot determine the parameter $p$ in Theorem \ref{Meier}, in the case of $f(u)=\frac{u^{2}+u}{2}$. From this point of view, we have to consider a new method, so called the minorant method to deal with a function $f$ which is not multiplicative. In conclusion, we provide a formula $\frac{f\left(\lVert S(t)u_{0}\rVert_{\infty}\right)}{\lVert S(t)u_{0}\rVert_{\infty}}$ instead of $\lVert S(t)u_{0}\rVert_{\infty}^{p-1}$ to give a criterion of the existence of the global solution when the source term is $\psi(t)f(u)$.

\section{Main results}\label{main}
In this section, we firstly introduce the minorant function and the majorant function. Next, we prove the main theorem by using the minorant function and majorant function.\\

First of all, we discuss multiplicative minorants and majorants of the function $f$, which will play an important role in this work.
\begin{definition}
	For a function $f$, the minorant function $f_{m}:[0,\infty)\rightarrow [0,\infty)$ and the majorant function $f_{M}:[0,\infty)\rightarrow [0,\infty)$ are defined by
	\[
	\begin{aligned}
		&f_{m}(u):=\inf_{0<\alpha<1}\frac{f(\alpha u)}{f(\alpha)},\,\,&u\geq 0,\\
		&f_{M}(u):=\sup_{0<\alpha<1}\frac{f(\alpha u)}{f(\alpha)},\,\,&u\geq 0.
	\end{aligned}
	\]
\end{definition}
Then  the following properties:
\begin{itemize}
	\item[$\cdot$] $f(\alpha)f_{m}(u)\leq f(\alpha u)\leq f(\alpha)f_{M}(u)$, $0<\alpha<1$, $u>0$.
	\item[$\cdot$] If $g$ and $h$ be functions satisfying that
	\[
	f(\alpha)g(u)\leq f(\alpha u)\leq f(\alpha)h(u),\,\,0<\alpha<1,\,\, u>0,
	\]
	then it follows that $g(u)\leq f_{m}(u)$ and $f_{M}(u)\leq h(u)$, $u>0$.
\end{itemize}
It follows that $f_{m}$ and $f_{M}$ are natural to call the multiplicative minorant and majorant of a function $f$ respectively. In fact, the values of $f_{m}$ and $f_{M}$ depend strongly on the value of $f$ near zero, since
\[
f_{m}(u)\leq \inf_{0<\alpha<\frac{1}{u}}\frac{f(\alpha u)}{f(\alpha)}\,\,\mbox{ and }\,\,f_{M}(u)\geq \sup_{0<\alpha<\frac{1}{u}}\frac{f(\alpha u)}{f(\alpha)},
\]
for each $u\geq 1$. Also, if $f$ is convex, then $\frac{f(u)}{u}$ is nondecreasing. Then it is easy to see that the function $f$, the minorant $f_{m}$, and the majorant $f_{M}$ satisfy the following properties:
\begin{itemize}
	\item[(i)] $f_{m}\left(u\right)\leq \frac{f(u)}{f(1)}\leq f_{M}\left(u\right)$ for $u\geq 0$.
	\item[(ii)] $f_{m}\left(\frac{1}{\alpha}\right)\leq \frac{f(1)}{f(\alpha)}\leq f_{M}\left(\frac{1}{\alpha}\right)$ for $0<\alpha\leq 1$.
	\item[(iii)] $f_{m}(1)=f_{M}(1)=1$.
	\item[(iv)] $\frac{f_{m}(u)}{u}$ and $\frac{f_{M}(u)}{u}$ are nondecreasing in $(0,1)$.
	\item[(v)] $f_{m}(u)\leq u$ and $f_{M}(u)\leq u$ for $0<u\leq 1$, since $\frac{f(\alpha u)}{f(\alpha)}=\frac{f(\alpha u)}{\alpha u}\frac{\alpha}{f(\alpha)}u\leq u$.
	\item[(vi)] $\int_{\eta}^{\infty}\frac{ds}{f_{M}(s)}\leq f(1) \int_{\eta}^{\infty}\frac{ds}{f(s)}\leq \int_{\eta}^{\infty}\frac{ds}{f_{m}(s)}$ for $\eta\geq 1$.
	\item[(vii)] $\int_{0}^{1}\frac{ds}{f_{m}(s)}=\int_{0}^{1}\frac{ds}{f(s)}=\int_{0}^{1}\frac{ds}{f_{M}(s)}=\infty$.
\end{itemize}
We obtain from the property $(vi)$ that $\int_{\eta}^{\infty}\frac{ds}{f(s)}<\infty$ implies $\int_{\eta}^{\infty}\frac{ds}{f_{M}(s)}<\infty$. However, the converse is not true, in general. In fact, examples and detailed properties the minorant function $f_{m}$ and the majorant function $f_{M}$ were discussed in \cite{CH}.\\

Now, we introduce the definition of the blow-up solutions and global solutions.
\begin{definition}
	We say that a solution $u$ blows up at finite time $t^{*}$, if there exists $0<t^{*}<\infty$ such that $\lVert u(\cdot,t)\rVert_{\infty}\rightarrow \infty$ as $t\rightarrow t^{*}$. On the other hand, a solution $u$ exists globally whenever $\lVert u(\cdot,t)\rVert_{\infty}$ is bounded for each time $t\geq 0$.
\end{definition}

Now, we prove Theorem \ref{mainthm1}.
\begin{proof}
	$(ii)\Leftrightarrow(iii)$ : It is well-known that $\lVert S(t)u_{0}\rVert_{\infty} \sim e^{-\lambda_{0} t}$ for $t\geq 0$, for every nonnegative and nontrivial initial data $u_{0}\in C_{0}(\Omega)$. Therefore, we easily see that for a nonnegative and nontrivial initial data $u_{0}\in C_{0}(\Omega)$, there exist poisitive constants $c_{1}$ and $c_{2}$ such that
	\[
	c_{1}e^{\lambda_{0} t}f(\epsilon e^{\lambda_{0} t})\leq\frac{f\left(\epsilon \lVert S(t)u_{0}\rVert_{\infty}\right)}{\lVert S(t)u_{0}\rVert_{\infty}} \leq c_{2}e^{\lambda_{0} t}f(\epsilon e^{\lambda_{0} t}),
	\]
	for each $\epsilon>0$, since $f$ is nondecreasing. Hence, by considering $w_{0}:=\epsilon u_{0}$, then we have
	\[
	\frac{c_{1}}{\epsilon} \int_{0}^{t}e^{\lambda_{0} \tau}f(\epsilon e^{\lambda_{0} \tau})d\tau \leq \int_{0}^{t}\frac{f\left( \lVert S(\tau)w_{0}\rVert_{\infty}\right)}{\lVert S(\tau)w_{0}\rVert_{\infty}}d\tau\leq \frac{c_{2}}{\epsilon} \int_{0}^{t}e^{\lambda_{0} \tau}f(\epsilon e^{\lambda_{0} \tau})d\tau,
	\]
	which completes the proof.\\
	
	$(ii)\Rightarrow (i)$ : Suppose that
	\[
	\int_{0}^{\infty}\psi(t)e^{\lambda_{0}t}f\left(\epsilon e^{-\lambda_{0}t}\right)dt=\infty
	\]
	for every $\epsilon>0$.	First of all, we consider the eigenfunction $\phi_{0}$ to be $\sup_{x\in\Omega}\phi_{0}dx=1$, corresponding to the first Dirichlet eigenvalue $\lambda_{0}$. Suppose that the solution $u$ exists globally, on the contrary. Multiplying the equation \eqref{mainequation} by $\phi_{0}$ and integrating over $\Omega$, we use Green's theorem and Jensen's inequality to obtain
	\[
	\begin{aligned}
		\int_{\Omega}u_{t}(x,t)\phi_{0}(x)dx=&\int_{\Omega}\phi_{0}(x)\Delta u(x,t)dx+\psi(t)\int_{\Omega}f(u(x,t))\phi_{0}(x)dx\\
		\geq &-\lambda_{0}\int_{\Omega}u(x,t)\phi_{0}(x)dx+\psi(t)f\left(\int_{\Omega}u(x,t)\phi_{0}(x)dx\right),
	\end{aligned}
	\]
	for all $t> 0$. Putting $y(t):=\int_{\Omega}u(x,t)\phi_{0}(x)dx$, for $t\geq 0$, then $y(t)$ exists for all time $t$ and satisfies the following inequality
	\[
	\begin{cases}
		y'(t)\geq -\lambda_{0} y(t)+\psi(t)f(y(t)),\,\,t>0,\\
		y(0)=y_{0}:=\int_{\Omega}u_{0}(x)\phi_{0}(x)dx>0.
	\end{cases}
	\]
	Then the inequality can be written as
	\begin{equation}\label{eq01}
		\left[e^{\lambda_{0}t}y(t)\right]'\geq \psi(t)e^{\lambda_{0}t}f(y(t))\geq 0,
	\end{equation}
	for $t>0$ so that $e^{\lambda_{0}t}y(t)$ is nondecreasing on $[0,\infty)$. On the other hand, by the definition of $f_{m}$, we can find $v_{1}\in[0,1]$ such that $f_{m}=0$ on $[0,v_{1})$ and $f_{m}>0$ on $(v_{1},\infty)$. Then there exists $\epsilon>0$ such that $y(0)>\epsilon v_{1}$. \textit{i.e.} $v_{1}<\frac{y(0)}{\epsilon}\leq \frac{e^{\lambda_{0}t}y(t)}{\epsilon}$ for $t\geq 0$. Combining all these arguments, it follows from \eqref{eq01} and the definition of $f_{m}$ that
	\[
	\frac{\frac{e^{\lambda_{0}t}y(t)}{\epsilon}}{f_{m}\left(\frac{e^{\lambda_{0}t}y(t)}{\epsilon}\right)}\geq \frac{1}{\epsilon}\psi(t)e^{\lambda_{0}t}f(\epsilon e^{-\lambda_{0}t}),
	\]
	for all $t>0$. Now, define a function $F_{m}:(v_{1},\infty)\rightarrow (0,v_{\infty})$ by
	\[
	F_{m}(v):=\int_{v}^{\infty}\frac{dw}{f_{m}(w)},\,\,v>v_{1}
	\]
	where $v_{\infty}:=\lim_{v\rightarrow v_{1}}\int_{v}^{\infty}\frac{dw}{f_{m}(w)}$. Then it is easy to see that $F_{m}$ is well-defined continuous function which is a strictly decreasing bijection with its inverse $F_{m}^{-1}$ and $\lim_{v\rightarrow\infty}F_{m}(v)=0$. Integrating the inequality \eqref{eq01} over $[0,t]$, we obtain
	\[
	F_{m}\left(\frac{y(0)}{\epsilon}\right)-F_{m}\left(\frac{e^{\lambda_{0}t} y(t)}{\epsilon}\right)\geq \frac{1}{\epsilon}\int_{0}^{t}\psi(\tau)e^{\lambda_{0}\tau}f(\epsilon e^{-\lambda_{0}\tau})d\tau,
	\]
	for all $t\geq 0$. Hence, we obtain
	\[
	y(t)\geq \epsilon e^{-\lambda_{0} t}F_{m}^{-1}\left[F_{m}\left(\frac{y(0)}{\epsilon}\right)- \frac{1}{\epsilon}\int_{0}^{t}\psi(\tau)e^{\lambda_{0}\tau}f(\epsilon e^{-\lambda_{0}\tau})d\tau \right]
	\]
	for all $t\geq 0$, which implies that $y(t)$ cannot be global.\\
	
	$(i)\Rightarrow (ii)$ : Suppose that
	\[
	\int_{0}^{\infty}\psi(t)e^{\lambda_{0}t}f\left(\epsilon e^{-\lambda_{0}t}\right)dt<\infty
	\]
	for some $\epsilon>0$. We note that there exists a maximal interval $[0,m^{*})$ on which $f_{M}$ is finite. Then it is true that the integral $\int_{v}^{m^{*}}\frac{dw}{f_{M}(w)}$ is finite for each $v\in (0,m^{*})$, $\lim_{v\rightarrow 0}\int_{v}^{m^{*}}\frac{dw}{f_{M}(w)}=\infty$, and $\lim_{v\rightarrow m^{*}}\int_{v}^{m^{*}}\frac{dw}{f_{M}(w)}=0$. Then a function $F_{M}:(0,m^{*})\rightarrow (0,\infty)$ defined by
	\[
	F_{M}(v):=\int_{v}^{m^{*}}\frac{dw}{f_{M}(w)},\,\,v\in(0,m^{*})
	\]
	is a well-defined continuous function which is a strictly decreasing bijection with its inverse $F_{M}^{-1}$. Now, take a number $z_{0}$ such that
	\[
	0<z_{0}<F_{M}^{-1}\left[\frac{1}{\epsilon}\int_{0}^{\infty}\psi(t)e^{\lambda_{0}t}f\left(\epsilon e^{-\lambda_{0}t}\right)dt\right]
	\]
	and define a nondecreasing function $z:[0,\infty)\rightarrow [z_{0},\infty)$ by
	\[
	z(t):=F_{M}^{-1}\left[F_{M}(z_{0})-\frac{1}{\epsilon}\int_{0}^{t}\psi(\tau)e^{\lambda_{0}\tau}f\left(\epsilon e^{-\lambda_{0}\tau}\right)d\tau\right],\,\,t\geq 0.
	\]
	Then $z(t)$ is a bounded solution of the following ODE problem:
	\[
	\begin{cases}
		z'(t)=\frac{1}{\epsilon}\psi(t)e^{\lambda_{0}t}f\left(\epsilon e^{-\lambda_{0}t}\right)f_{M}\left(z(t)\right),\,\,t>0,\\
		z(0)=z_{0}.
	\end{cases}
	\]Now, consider a function $v(x,t):=e^{-\lambda_{0}t}\phi_{0}(x)$ on $\overline{\Omega}\times[0,\infty)$ which is a solution to the heat equation $v_{t}=\Delta v$ under the Dirichlet boundary condition. Let $\overline{u}(x,t):=z(t)v(x,t)$ for $(x,t)\in\overline{\Omega}\times[0,\infty)$. Since $f$ is convex, $\frac{f(u)}{u}$ is nondecrasing. Then it follows that
	\[
	\begin{aligned}
		\overline{u}_{t}(x,t)=&\Delta \overline{u}(x,t)+\psi(t)v(x,t)\left[\frac{f(\epsilon e^{-\lambda_{0}t})}{\epsilon e^{-\lambda_{0}t}}\right]f_{M}(z(t))\\
		\geq&\Delta \overline{u}(x,t)+\psi(t)v(x,t)\left[\frac{f(v(x,t))}{v(x,t)}\right]f_{M}(z(t))\\
		\geq&\Delta \overline{u}(x,t)+\psi(t)f( \overline{u}(x,t))
	\end{aligned}
	\]
	for all $(x,t)\in\Omega\times(0,\infty)$. It follows that $\overline{u}$ is the supersolution to the equation \eqref{mainequation}, which implies that $u$ exists globally.
\end{proof}
\begin{Cor}\label{corollary1}
	Let the function $\psi$ be a nonnegative continuous function and the function $f$ be a nonnegative continuous and quasi-multiplitive function. \textit{i.e.} there exist $\gamma_{2}\geq \gamma_{1}>0$ such that
	\begin{equation}\label{multiplicative}
		\gamma_{1} f(\alpha)f(u)\leq f(\alpha u)\leq \gamma_{2} f(\alpha)f(u),
	\end{equation}
	for $0<\alpha<1$ and $u>0$. Then the following statements are equivalent:
	\begin{itemize}
		\item[(i)] $\int_{0}^{\infty}\psi(t)e^{\lambda_{0}t}f\left(e^{-\lambda_{0}t}\right)dt=\infty$.
		\item[(ii)] $\int_{0}^{\infty}\psi(t)\frac{f\left(\lVert S(t)w_{0}\rVert_{\infty}\right)}{\lVert S(t)w_{0}\rVert_{\infty}}dt=\infty$ for every nonnegative and nontrivial $w_{0}\in C_{0}(\Omega)$.
		\item[(iii)] $
		\int_{0}^{\infty}\psi(t)\frac{dt}{F\left(e^{-\lambda_{0}t}\right)}=\infty,
		$
		where $F(v):=\int_{v}^{\infty}\frac{dw}{f(w)}$.
		\item[(iv)] There is no global solution to the equation \eqref{mainequation} for any initial data.
	\end{itemize}
\end{Cor}
\begin{proof}
	Theorem \ref{mainthm1} says that (i), (ii), and (iv) are equivalent. Therefore, we now discuss (iii).\\
	\textbf{(i)$\Leftrightarrow$(iii)} : Let $F(v)=\int_{v}^{\infty}\frac{dw}{f(w)}$. Then the assumption \eqref{multiplicative} follows that
	\[
	\int_{1}^{\infty}\frac{z}{\gamma_{2}f(z)f(s)}ds\leq \int_{z}^{\infty}\frac{dw}{f(w)}\leq \int_{1}^{\infty}\frac{z}{\gamma_{1}f(z)f(s)}ds,
	\]
	for $z>0$. This implies that
	\[
	\frac{F(1)}{\gamma_{2}}\frac{z}{f(z)}\leq F(z)\leq \frac{F(1)}{\gamma_{1}}\frac{z}{f(z)}.
	\]
	\textit{i.e.} $F(z)\sim \frac{z}{f(z)}$, $z>0$. Therefore, the proof is complete.
\end{proof}
\begin{remark}
	In 2014, Loayza and Paix\~{a}o \cite{LP} studied the conditions for existence and nonexistence of the global solutions to the equation \eqref{mainequation} under the general domain and obtained the following statements:
	\begin{itemize}
		\item[(i)] for every $w_{0}\in C_{0}(\Omega)$, there exist $\tau>0$ such that
		\begin{equation}\label{LPcon1}
			\int_{\lVert S(\tau)w_{0}\rVert_{\infty}}^{\infty}\frac{dw}{f(w)}\leq \int_{0}^{\tau}\psi(\sigma)d\sigma,
		\end{equation}
		then there is no global solution $u$ for every initial data,
		\item[(ii)] the solution $u$ exists globally for small initial data, whenever
		\begin{equation}\label{LPcon2}
			\int_{0}^{\infty}\psi(t)\frac{f(\lVert S(t)w_{0}\rVert_{\infty})}{\lVert S(t)w_{0}\rVert_{\infty}}dt <1,
		\end{equation}
		for some $w_{0}\in C_{0}(\Omega)$.
	\end{itemize}
	In fact, Corollary \ref{corollary1} imply that the conditions \eqref{LPcon1} and \eqref{LPcon2} have a strong relation, even though \eqref{LPcon1} and \eqref{LPcon2} have different formulas.
\end{remark}
Also, by using Corollary \ref{corollary1}, the example in Section \ref{discussion} can be characterized completely as follows:
\begin{example}
	Let the domain $\Omega$ be bounded in $\mathbb{R}^{N}$, $\psi(t):=(t+1)^{-\sigma}e^{kt}$, and $f(u):=u^{p}$ where $\sigma\in\mathbb{R}$, $k\in\mathbb{R}$, and $p>1$. Then the following statements are true.
	\begin{itemize}
		\item[(i)] If $k>(p-1)\lambda_{0}$, then there is no global solution $u$ to the equation \eqref{mainequation} for any nonnegative and nontrivial initial data $u_{0}\in C_{0}(\Omega)$.
		\item[(ii)] If $k<(p-1)\lambda_{0}$, then there exists a global solution to the equation \eqref{mainequation} for sufficiently small initial data $u_{0}\in C_{0}(\Omega)$.
		\item[(iii)] If $k=(p-1)\lambda_{0}$ and $\sigma\leq 1$, then there is no global solution $u$ to the equation \eqref{mainequation} for any nonnegative and nontrivial initial data $u_{0}\in C_{0}(\Omega)$.
		\item[(iv)] If $k=(p-1)\lambda_{0}$ and $\sigma>1$, then there exists a global solution to the equation \eqref{mainequation} for sufficiently small initial data $u_{0}\in C_{0}(\Omega)$.
	\end{itemize}
\end{example}

\section*{Conflict of Interests}
\noindent The authors declare that there is no conflict of interests regarding the publication of this paper.

\section*{Acknowledgments}
\noindent This work was supported by the National Research Foundation of Korea (NRF) grant funded by the Korea government (MSIT) (No. 2021R1A2C1005348).


\begin{thebibliography}{20}
\bibitem{AW}
D. G. Aronson and H. F. Weinberger, \textit{Multidimensional nonlinear diffusion arising in population genetics}, Adv. in Math. 30 (1978), no. 1, 33–76.

\bibitem{BZW}
X. Bai, S. Zheng, and W. Wang, \textit{Critical exponent for parabolic system with time-weighted sources in bounded domain}, J. Funct. Anal. 265 (2013), no. 6, 941–952.

\bibitem{CH}
S. -Y. Chung and J. Hwang, \textit{A necessary and sufficient condition for the existence of global solutions to discrete semilinear parabolic equations on networks}, Chaos Solitons Fractals 158 (2022), Paper No. 112055.

\bibitem{CL1}
R. Castillo and M. Loayza, \textit{On the critical exponent for some semilinear reaction–diffusion systems on general domains}, J. Math. Anal. Appl 428(2015), 1117–1134.

\bibitem{CL2}
R. Castillo and M. Loayza, \textit{Global existence and blow up for a coupled parabolic system with time-weighted sources on a general domain}, Z. Angew. Math. Phys. 70(2019), Art.57, 16pp.

\bibitem{DL}
K. Deng, H.A. Levine, \textit{The role of critical exponents in blow-up theorems: The sequel}, J. Math. Anal. Appl. 243 (2000) 85–126.

\bibitem{F}
H. Fujita, \textit{On the blowing up of solutions of the Cauchy problems for $u_{t}=\Delta u+u^{1+\alpha}$}, J. Fac, Sci. Univ. Tokyo Sect. IA 13(1966), 109-124


\bibitem{Fr}
A. Friedman, \textit{Partial differential equations}, Holt, Rinehart and Winston, Inc., New York-Montreal, Que.-London, 1969.

\bibitem{Ha}
K. Hayakawa, \textit{On the nonexistence of global solutions of some semilinear parabolic equations}, Proc. Jpn. Acad. 49 (1973), 503–525.

\bibitem{H}
B. Hu, \textit{Blow-up Theories for Semilinear Parabolic Equations}, Lecture Notes in Mathematics, 2018. Springer, Heidelberg, 2011.

\bibitem{L}
H. A. Levine, \textit{The role of critical exponents in blow-up theorems}, SIAM Rev. 32 (2) (1990) 262–288.

\bibitem{LP}
M. Loayza and C. S. da Paix\~{a}o, \textit{Existence and non-existence of global solutions for a semilinear heat equation on a general domain}, Electron. J. Differential Equations 2014, No. 168, 9 pp.

\bibitem{M}
P. Meier, \textit{On the critical exponent for reaction-diffusion equations}, Arch. Rational Mech. Anal. 109(1990), 63-71.

\bibitem{Q}
Y. Qi, \textit{The critical exponents of parabolic equations and blow-up in $\mathbb{R}^{N}$}, Proc. Roy. Soc. Edinburgh 128A (1998), 123–136.

\bibitem{QW}
Y. -W. Qi and M. -X. Wang, \textit{Critical exponents of quasilinear parabolic equations}, J. Math. Anal. Appl. 267 (2002), no. 1, 264–280.
\end{thebibliography}
\end{document}